\documentclass[10pt]{amsart}
\usepackage{amssymb,amstext,amsmath,amscd,amsthm,amsfonts,enumerate,latexsym,stmaryrd,multicol,geometry,graphicx,mathrsfs}
\usepackage[usenames]{color}
\usepackage[all]{xy}
\newtheorem{thm}{Theorem}[section]
\newtheorem{lem}[thm]{Lemma}
\newtheorem{prop}[thm]{Proposition}
\newtheorem{cor}[thm]{Corollary}
\theoremstyle{definition}
\newtheorem{dfn}[thm]{Definition}
\newtheorem{ques}[thm]{Question}

\newtheorem{rem}[thm]{Remark}
\newtheorem{conv}[thm]{Convention}

\newtheorem{ex}[thm]{Example}

\theoremstyle{remark}

\newtheorem*{claim*}{Claim}
\newtheorem*{ac}{Acknowlegments}

\numberwithin{equation}{thm}
 \def\add{\operatorname{add}}
\def\db{\mathsf{D^b}}
\def\dbe{\mathsf{D}^\mathsf{b}_E}
\def\dpf{\mathsf{D^{pf}}}
\def\dpfe{\mathsf{D}^\mathsf{pf}_E}
\def\ds{\mathsf{D^{sg}}}
\def\dse{\mathsf{D}^\mathsf{sg}_E}
\def\E{\mathcal{E}}
\def\End{\operatorname{End}}
\def\Ext{\operatorname{Ext}}
\def\h{\operatorname{H}}
\def\Hom{\operatorname{Hom}}

\def\ipd{\operatorname{IPD}}
\def\lten{\otimes^{\mathbf{L}}}
\def\m{\mathfrak{m}}

\def\p{\mathfrak{p}}

\def\pd{\operatorname{pd}}
\def\rhom{\operatorname{\mathbf{R}Hom}}

\def\spcl{\operatorname{spcl}}
\def\spec{\operatorname{Spec}}
\def\ssupp{\operatorname{Supp^{sg}}}
\def\supp{\operatorname{Supp}}
\def\sing{\operatorname{Sing}}
\def\T{\mathcal{T}}

\def\thick{\operatorname{thick}}

\def\U{\mathcal{U}}
\def\V{\operatorname{V}}
\def\X{\mathcal{X}}
\def\xx{\boldsymbol{x}}
\def\Y{\mathcal{Y}}
\def\Z{\mathbb{Z}}
\def\ZZ{\mathcal{Z}}
\begin{document}
\title[Supports of intersections of thick subcategories]{Supports of intersections of thick subcategories\\
and proxy smallness}
\author{Kiriko Kato}
\address{Department of Mathematics, Graduate School of Science, Osaka Metropolitan University, 3-3-138, Sugimoto, Sumiyoshi, Osaka 558-8585, Japan}
\email{kiriko.kato@omu.ac.jp}
\author{Ryo Takahashi}
\address{Graduate School of Mathematics, Nagoya University, Furocho, Chikusaku, Nagoya 464-8602, Japan}
\email{takahashi@math.nagoya-u.ac.jp}
\urladdr{https://www.math.nagoya-u.ac.jp/~takahashi/}
\subjclass[2020]{13D09, 13H10}
\keywords{derived category, thick subcategory, support, perfect complex, proxy small subcategory, singularity category, singular support, complete intersection, hypersurface, locally dominant ring}
\thanks{Takahashi was partly supported by JSPS Grant-in-Aid for Scientific Research 23K03070}
\begin{abstract}
Let $R$ be a commutative noetherian ring.
Let $\db(R)$ be the bounded derived category of finitely generated $R$-modules.
Let $\X$ and $\Y$ be thick subcategories of $\db(R)$.
In this paper, we consider the question asking when the equality $\supp(\X\cap\Y)=\supp\X\cap\supp\Y$ holds, and give several answers.
As applications, we obtain a characterization of the proxy small subcategories, and classifications of certain thick subcategories.
\end{abstract}
\maketitle
\section{Introduction}

Let $R$ be a commutative noetherian ring.
Let $\db(R)$ be the bounded derived category of finitely generated $R$-modules, and $\dpf(R)$ the full subcategory of $\db(R)$ consisting of perfect complexes.
In the present paper, we raise and deal with the following quite natural question.

\begin{ques}\label{13}
Let $\X,\Y$ be thick subcategories of $\db(R)$.
When does the following equality hold true?
\begin{equation}\label{14}
\supp(\X\cap\Y)=\supp\X\cap\supp\Y.
\end{equation}
\end{ques}

Following \cite{proxy}, we say that a thick subcategory $\ZZ$ of $\db(R)$ is {\em proxy small} if there is an equality $\supp\ZZ=\supp(\ZZ\cap\dpf(R))$.
A typical example of a proxy small subcategory of $\db(R)$ is a thick subcategory of $\dpf(R)$.
The notion of proxy small subcategories corresponds to that of a proxy small complex introduced by Dwyer, Greenlees and Iyengar \cite{DGI} in the way that the proxy small complex is the complex whose generating thick subcategory is a proxy small subcategory.
Note by definition that, for $\Y=\dpf(R)$, those thick subcategories $\X$ of $\db(R)$ which satisfy the equality \eqref{14} are precisely the proxy small subcategories of $\db(R)$.

Applying the classical Hopkins--Neeman theorem \cite{H} and the relatively recent theorem of Pollitz \cite{P}, we are able to get several answers to Question \ref{13}.

\begin{thm}[Theorem \ref{25}]\label{22}
\begin{enumerate}[\rm(1)]
\item
The equality \eqref{14} holds provided that both $\X$ and $\Y$ are proxy small.
\item
Suppose that $R$ is a local ring.
Then the following three conditions are equivalent.
\begin{itemize}
\item
The ring $R$ is a complete intersection.
\item
The equality \eqref{14} holds for all thick subcategories $\X$ and $\Y$ of $\db(R)$.
\item
The equality \eqref{14} holds for all thick subcategories $\X$ of $\db(R)$ and $\Y=\dpf(R)$.
\end{itemize}
\end{enumerate}
\end{thm}

In the case where $\Y$ is contained in $\dpf(R)$, we obtain a criterion for the equality \eqref{14} to hold.

\begin{thm}[Theorem \ref{21}]\label{9}
Let $\X$ and $\Y$ be thick subcategories of $\db(R)$.
Suppose that $\Y$ is contained in $\dpf(R)$.
Then the equality \eqref{14} holds if and only if the following full subcategory of $\db(R)$ is triangulated.
$$
\X\ast\Y=\left\{T\in\db(R)\,\bigg|\,\begin{matrix}\text{there is an exact triangle $X\to T\to Y\rightsquigarrow$}\\
\text{in $\db(R)$ with $X\in\X$ and $Y\in\Y$}\end{matrix}\right\}\!.
$$
\end{thm}

Restricting the above theorem to the case where $\Y=\dpf(R)$, we get a characterization of the proxy small subcategories.
This complements a result in \cite{proxy}, which essentially shows the ``only if'' part.

\begin{cor}[Corollary \ref{2}]\label{10}
A thick subcategory $\X$ is proxy small if and only if $\X\ast\dpf(R)$ is triangulated.
\end{cor}

Following \cite{dlr}, we say that $R$ is locally dominant if for each prime ideal $\p$ and for every nonperfect $R_\p$-complex $X$, the residue field $\kappa(\p)$ belongs to the thick subcategory of $\db(R_\p)$ generated by $R_{\p} $ and $X$.
A ring which is locally a hypersurface (e.g., a local hypersurface) is a typical example of a locally dominant ring.

As an application of the above corollary, we get one-to-one correspondences between certain thick subcategories of $\db(R)$ and certain sets of prime ideals of $R$ when $R$ is locally dominant.

\begin{cor}[Theorems \ref{3} and \ref{6}]\label{11}
Suppose that $R$ is a locally dominant ring.
Let $\E$ be a thick subcategory of $\dpf(R)$, and set $E=\supp\E$.
Then there are one-to-one correspondences
$$
\left\{\begin{matrix}
\text{proxy small}\\
\text{subcategories}\\
\text{$\X$ of $\db(R)$ with}\\
\text{$\X\cap\dpf(R)=\E$}
\end{matrix}\right\}
=
\left\{\begin{matrix}
\text{thick subcategories}\\
\text{$\X$ of $\db(R)$}\\
\text{with $\E\subseteq\X$ and}\\
\text{$E=\supp\X$}
\end{matrix}\right\}
\cong
\left\{\begin{matrix}
\text{thick subcategories}\\
\text{$\Y$ of $\db(R)$}\\
\text{with $R\in\Y$ and}\\
\text{$\ipd(\Y)\subseteq E$}
\end{matrix}\right\}
\cong
\left\{\begin{matrix}
\text{specialization-closed}\\
\text{subsets of $\sing R$}\\
\text{contained in $E$}
\end{matrix}\right\}\!.
$$
\end{cor}

\noindent
Here, $\sing R$ stands for the singular locus of $R$, while $\ipd(\Y)$ denotes the infinite projective dimension locus of $\Y$, that is, the set of prime ideals $\p$ of $R$ such that $\pd_{R_\p}Y_\p=\infty$ for some object $Y\in\Y$.
It would be worth emphasizing that the mutually inverse bijections in the above corollary are explicitly described.

The organization of this paper is as follows.
In Section \ref{7}, we recall basic notions, fundamental properties and some results for later use.
In Section \ref{12}, we provide several observations about Question \ref{13}, and prove Theorem \ref{22}.
In this section, we also deal with the singularity category version of \eqref{14}, i.e., the equality
$$
\ssupp(\X\cap\Y)=\ssupp\X\cap\ssupp\Y
$$
of singular supports of thick subcategories $\X,\Y$ of the singularity category $\ds(R)$ of $R$.
In Section \ref{8}, we consider characterizing the thick subcategories that satisfy \eqref{14}, and prove Theorem \ref{9} and Corollary \ref{10}.
In Section \ref{19}, we provide a lot of applications of Corollary \ref{10}.
In this section, we show Corollary \ref{11}.

\section{Preliminaries}\label{7}

This section is devoted to preliminary definitions and results, which will be used in the later sections.
Let us begin with stating our convention, which is valid throughout the present paper.

\begin{conv}
Throughout this paper, we assume that all subcategories are strictly full subcategories, and that all rings are commutative noetherian rings with identity.
Let $R$ be a (commutative noetherian) ring.
We denote by $\ds(R)$ the {\em singularity category} of $R$, that is to say, the Verdier quotient of the triangulated category $\db(R)$ by the triangulated subcategory $\dpf(R)$.
Also, $\sing R$ stands for the {\em singular locus} of $R$, namely, the set of prime ideals $\p$ of $R$ such that the localization $R_\p$ of $R$ at $\p$ is not a regular local ring.
\end{conv}

We recall some notions of subcategories of a triangulated category.

\begin{dfn}
\begin{enumerate}[(1)]
\item
A {\em thick} subcategory of a triangulated category $\T$ is by definition a triangulated subcategory of $\T$ closed under direct summands.
For an object $T\in\T$, we denote by $\thick_\T T$ the {\em thick closure} of $T$ in $\T$, that is, the smallest thick subcategory of $\T$ containing $X$.
Note that $\thick_{\db(R)}R=\dpf(R)$.
\item
Let $F:\T\to\U$ be an exact functor of triangulated categories.
Let $\X,\Y$ be subcategories of $\T,\U$ respectively.
The {\em essential image} $F\X$ of $\X$ by $F$ and the {\em inverse image} $F^{-1}\Y$ of $\Y$ by $F$ are defined by:
$$
F\X=\{FX\mid X\in\X\},\qquad
F^{-1}\Y=\{T\in\T\mid FT\in\Y\}.
$$
The essential image $F\X$ is a subcategory of $\U$, while the inverse image $F^{-1}\Y$ is a subcategory of $\T$.
Note that $F\X$ is closed under shifts when so is $\X$, and that $F^{-1}\Y$ is thick when so is $\Y$.
\item
Let $\T$ be a triangulated category.
Let $\X$ and $\Y$ be two subcategories of $\T$.
We define the {\em extension} of $\X$ and $\Y$ by the following equality; it is a subcategory of $\T$.
$$
\X\ast\Y=\{T\in\T\mid\text{there exists an exact triangle $X\to T\to Y\rightsquigarrow$ in $\T$ with $X\in\X$ and $Y\in\Y$}\}.
$$
\end{enumerate}
\end{dfn}

Next we recall the notions of supports and their variants.

\begin{dfn}
\begin{enumerate}[(1)]
\item
Let $T$ be an object $T$ of $\db(R)$ and $\X$ a subcategory of $\db(R)$.
We denote by $\supp T$ and $\supp\X$ the {\em supports} of $T$ and $\X$, that is,
$$
\supp T=\{\p\in\spec R\mid T_\p\ne0\text{ in }\db(R_\p)\},\qquad
\textstyle\supp\X=\bigcup_{X\in\X}\supp X.
$$
Also, we denote by $\ipd(T)$ and $\ipd(\X)$ the {\em infinite projective dimension loci} of $T$ and $\X$, that is,
$$
\ipd(T)=\{\p\in\spec R\mid T_{\p}\notin\dpf(R_{\p}) \},\qquad
\textstyle\ipd(\X)=\bigcup_{X\in\X}\ipd(X).
$$
\item
For an object $T$ of $\ds(R)$ and a subcategory of $\X$ of $\ds(R)$, we denote by $\ssupp T$ and $\ssupp\X$ the {\em singular supports} of $T$ and $\X$, that is,
$$
\ssupp T=\{\p\in\spec R\mid T_\p\ne0\text{ in }\ds(R_\p)\},\qquad
\textstyle\ssupp\X=\bigcup_{X\in\X}\ssupp X.
$$
Note that for an object $T$ of $\db(R)$ and a subcategory $\X$ of $\db(R)$ one has
$$
\ssupp(\pi T)=\ipd(T),\qquad\ssupp(\pi\X)=\ipd(\X),
$$
where $\pi:\db(R)\to\ds(R)$ denotes the canonical functor.
\end{enumerate}
\end{dfn}

The celebrated result below by Hopkins \cite{H} and Neeman \cite[Theorem 1.5]{N} will often be used in this paper.

\begin{thm}[Hopkins--Neeman]\label{17}
The assignments $\X\mapsto\supp\X$ and $\{X\in\db(R)\mid\supp X\subseteq W\}\mapsfrom W$ give mutually inverse bijections
$$
\{\text{thick subcategories of $\dpf(R)$}\}
\rightleftarrows\spcl(\spec R).
$$
\end{thm}

In \cite{proxy} the following notion has been introduced, which will play a main role in this paper.

\begin{dfn}
A thick subcategory $\X$ of $\db(R)$ is called {\em proxy small} if $\supp\X=\supp(\X\cap\dpf(R))$.
\end{dfn}

The remark below contains the reason why such a subcategory is called proxy small.

\begin{rem}
An object $X$ of $\db(R)$ is {\em proxy small} in the sense of Dwyer, Greenlees and Iyengar \cite{DGI} if and only if $\thick_{\db(R)}X$ is a proxy small subcategory of $\db(R)$.
For the details, see \cite[Proposition 3.3]{proxy}.
\end{rem}

The following theorem is given in \cite[Proposition 3.4]{proxy}, which is essentially due to Pollitz \cite[Theorem 5.2]{P}.

\begin{thm}[Pollitz]\label{23}
Suppose that the ring $R$ is local.
Then $R$ is a complete intersection if and only if every thick subcategory of $\db(R)$ is proxy small.
\end{thm}

Let us recall the definition of specialization-closed subsets, which are basic tools to classify subcategories.

\begin{dfn}
Let $X$ be a topological space.
A subset $A$ of $X$ is called {\em specialization-closed} if for every $a\in A$ the closure $\overline{\{a\}}$ is contained in $A$.
Denote by $\spcl X$ the set of specialization-closed subsets of $X$.
Note that the specialization-closed subsets of $X$ are precisely the (possibly infinite) unions of closed subsets of $X$.
\end{dfn}

Now we recall the definition of locally dominant rings, which frequently appear in this paper.

\begin{dfn}
\begin{enumerate}[(1)]
\item
Suppose that $R$ is a local ring.
We say that $R$ is {\em dominant} if, for every nonzero object $X$ of $\ds(R)$, the residue field of $R$ belongs to the thick closure $\thick_{\ds(R)}X$ of $X$ in $\ds(R)$. 
\item
We say that $R$ is {\em locally dominant} if for every prime ideal $\p$ of $R$ the local ring $R_\p$ is dominant.
\end{enumerate}
A typical example of a locally dominant ring is a ring which is locally a hypersurface; see \cite[Proposition 5.3]{dlr}.
In particular, every local ring that is a hypersurface is a locally dominant ring.
\end{dfn}

To state the next theorem, we prepare some notations.

\begin{dfn}
Let $\X$ be a proxy small subcategory of $\db(R)$.
Then, in view of \cite[Lemma 4.3]{proxy}, the Verdier quotient $\X/\X\cap\dpf(R)$ is naturally viewed as a full triangulated subcategory of $\ds(R)$, i.e., $\X/\X\cap\dpf(R)=\pi\X$, where $\pi:\db(R)\to\ds(R)$ stands for the canonical functor.
We set
$$
\sigma(\X)=(\supp\X,\,\ssupp(\X/\X\cap\dpf(R)))=(\supp\X,\,\ssupp(\pi\X)).
$$
Let $(V,W)$ be a pair of sets $V,W$ of prime ideals of $R$.
We define the subcategory $\tau(V,W)$ of $\db(R)$ by
$$
\tau(V,W)=\{X\in\db(R)\mid\supp X\subseteq V\text{ and }\ssupp(\pi X)\subseteq W\}.
$$
\end{dfn}

The following theorem is the combination of the main results of \cite{dlr,proxy}.
For the details, we refer the reader to \cite[Remark 10.13]{dlr} and \cite[Theorem 4.6]{proxy}.

\begin{thm}[Takahashi]\label{5}
Let $R$ be locally dominant.
Then the following assertions hold true.
\begin{enumerate}[\rm(1)]
\item
The assignments $\X\mapsto\ipd(\X)$ and $\{X\in\db(R)\mid\ipd(X)\subseteq W\}\mapsfrom W$ give mutually inverse bijections
$$
\{\text{thick subcategories of $\db(R)$ containing $R$}\}
\rightleftarrows\spcl(\sing R).
$$
\item
The assignments $\sigma,\tau$ give mutually inverse bijections
$$
\xymatrix{
{\left\{\begin{matrix}\text{proxy small}\\
\text{subcategories of $\db(R)$}\end{matrix}\right\}}
\ar@<.7mm>[r]^-\sigma& 
{\left\{(V,W)\,\bigg|\,\begin{matrix}\text{$V\in\spcl(\spec R)$ and}\\
\text{$W\in\spcl(\sing R)$ with $V\supseteq W$}
\end{matrix}\right\}\!.}
\ar@<.7mm>[l]^-\tau}
$$
\end{enumerate}
\end{thm}

\section{Cases where the support equality holds or not}\label{12}

In this section, we seek answers to Question \ref{13}, applying results from the literature.
We present examples where the equality \eqref{14} holds and ones where it fails.
We begin with remarking that the definition of proxy smallness is interpreted as a special case of the equality \eqref{14}.

\begin{rem}\label{15}
In the case where $\Y=\dpf(R)$, the equality \eqref{14} holds if and only if the subcategory $\X$ is proxy small.
This follows from the definition of proxy smallness and the fact that $\supp\dpf(R)=\spec R$.
\end{rem}

The following proposition is an immediate consequence of Remark \ref{15} and Theorem \ref{23}.

\begin{prop}\label{24}
Let $R$ be a local ring, and let $\Y=\dpf(R)$.
Then $R$ is a complete intersection if and only if the equality \eqref{14} holds for all thick subcategories $\X$ of the derived category $\db(R)$.
\end{prop}

The proposition below gives rise to a condition for \eqref{14} not to be satisfied.

\begin{prop}\label{16}
Let $R$ be a non-Gorenstein local ring.
Then the following two statements hold true.
\begin{enumerate}[\rm(1)]
\item
Let $\X,\Y$ be nonzero thick subcategories of the derived category $\db(R)$.
Assume that every object in $\X$ has finite injective dimension and every object in $\Y$ has finite projective dimension.
Then \eqref{14} fails.
\item
Let $\X$ be the subcategory of $\db(R)$ consisting of complexes of finite injective dimension.
Then the equality \eqref{14} fails for every nonzero thick subcategory $\Y$ of $\dpf(R)$.
\end{enumerate}
\end{prop}

\begin{proof}
(1) As $\X$ and $\Y$ are nonzero, the maximal ideal $\m$ of $R$ belongs to both $\supp\X$ and $\supp\Y$.
Suppose that \eqref{14} holds.
Then $\m$ belongs to $\supp(\X\cap\Y)$, which means $\X\cap\Y$ is nonzero.
We can choose a nonzero object $Z\in\X\cap\Y$.
By assumption, $Z$ has finite projective dimension and finite injective dimension.
It follows from \cite[Theorem 5.3]{DGI} that the local ring $R$ is Gorenstein.
This contradiction completes the proof.

(2) In view of (1), it is enough to verify that $\X$ is nonzero.
Let $K$ be the Koszul complex of a system of generators of the maximal ideal $\m$ of $R$.
Let $I$ be the Matlis dual of $K$.
Then each component of the bounded complex $I$ is an injective $R$-module, so $I$ has finite injective dimension.
Each cohomology of the complex $I$ is a finite-dimensional vector space over the residue field $R/\m$, and in particular, it is a finitely generated $R$-module.
Thus $I$ is isomorphic to an object $J$ of $\db(R)$.
Therefore $\X$ contains the nonzero object $J$.
\end{proof}

So far, whenever we consider a pair of thick subcategories of $\db(R)$ that does not satisfy \eqref{14}, either of those two subcategories is contained in $\dpf(R)$.
Let us construct an example of a pair of thick subcategories of $\db(R)$ which neither satisfies \eqref{14} nor fits inside $\dpf(R)$.

\begin{ex}
Let $R=k[\![x,y,z]\!]/(x^2,xy,y^2)$ with $k$ a field.
Then $R$ is a complete Cohen--Macaulay non-Gorenstein local ring of dimension one, and in particular, it possesses a canonical module $\omega$.
Let $\m=(x,y,z)$ be the maximal ideal of $R$, and let $\p=(x,y)$ be the minimal prime ideal of $R$.
Put $\X=\thick_{\db(R)}\omega$ and $\Y=\thick_{\db(R)}\m$.
As the $R$-modules $\omega$ and $\m$ have infinite projective dimension, it is observed that neither $\X$ nor $\Y$ is contained in $\dpf(R)$.
The localization $\omega_\p$ is the injective hull of the residue field of the artinian local ring $R_\p$, while $\m_\p=R_\p$.
In particular, neither $\omega_\p$ nor $R_\p$ is zero, and hence $\p$ belongs to $\supp\X\cap\supp\Y$.

Suppose that $\p$ belongs to $\supp(\X\cap\Y)$.
Then there exists an $R$-complex $Z\in\X\cap\Y$ such that $Z_\p\ncong0$.
As the $R_\p$-module $\omega_\p$ is injective, every object in $\thick_{\db(R_\p)}\omega_\p$ has finite injective dimension over $R_\p$.
Also, every object in $\thick_{\db(R_\p)}R_\p$ has finite projective dimension over $R_\p$.
It is observed that the $R_\p$-complex $Z_\p$ has finite projective dimension and finite injective dimension, which forces the local ring $R_\p$ to be Gorenstein by \cite[Theorem 5.3]{DGI}.
However, it is easy to see that $R_\p\cong k[\![x,y,z]\!]_{(x,y)}/(x^2,xy,y^2)$ is not Gorenstein.
This contradiction shows that $\p$ is not in $\supp(\X\cap\Y)$, and we conclude that $\supp(\X\cap\Y)\ne\supp\X\cap\supp\Y$.
\end{ex}

Now we apply the Hopkins--Neeman theorem to get an answer to Question \ref{13}.

\begin{prop}\label{18}
The equality \eqref{14} holds if both the subcategories $\X,\Y$ of $\db(R)$ are contained in $\dpf(R)$.
\end{prop}

\begin{proof}
Pick any element $\p\in\supp\X\cap\supp\Y$.
Let $K$ be the Koszul complex of a system of generators of the prime ideal $\p$.
Using \cite[Proposition 2.3(3)]{dm}, we see that $\supp K=\V(\p)\subseteq\supp\X\cap\supp\Y$.
As $K$ belongs to $\dpf(R)$ and $\X,\Y$ are contained in $\dpf(R)$, Theorem \ref{17} implies that $K$ belongs to $\X\cap\Y$.
It follows that $\p\in\V(\p)=\supp K\subseteq\supp(\X\cap\Y)$.
We conclude the equality $\supp(\X\cap\Y)=\supp\X\cap\supp\Y$ holds.
\end{proof}

Using the above proposition, we obtain a sufficient condition for the equality \eqref{14} to hold, which leads us to a characterization of the complete intersection local rings in terms of the equality \eqref{14}.

\begin{thm}\label{25}
\begin{enumerate}[\rm(1)]
\item
The equality \eqref{14} holds whenever the thick subcategories $\X,\Y$ are proxy small.
\item
Suppose that $R$ is a local ring.
Then the following three conditions are equivalent.
\begin{enumerate}[\rm(a)]
\item
The local ring $R$ is a complete intersection.
\item
The equality \eqref{14} holds for all thick subcategories $\X,\Y$ of $\db(R)$.
\item
The equality \eqref{14} holds for all thick subcategories $\X$ of $\db(R)$ and $\Y=\dpf(R)$.
\end{enumerate}
\end{enumerate}
\end{thm}

\begin{proof}
(1) As $\X\cap\dpf(R)$ and $\Y\cap\dpf(R)$ are thick subcategories contained in $\dpf(R)$, Proposition \ref{18} yields
$$
\supp(\X\cap\dpf(R))\cap\supp(\Y\cap\dpf(R))
=\supp((\X\cap\dpf(R))\cap(\Y\cap\dpf(R)))
=\supp(\X\cap\Y\cap\dpf(R)).
$$
Since both $\X$ and $\Y$ are proxy small subcategories of the derived category $\db(R)$, we get equalities
$$
\supp\X\cap\supp\Y
=\supp(\X\cap\dpf(R))\cap\supp(\Y\cap\dpf(R))
=\supp(\X\cap\Y\cap\dpf(R)).
$$
Note that $\supp(\X\cap\Y\cap\dpf(R))$ is contained in $\supp(\X\cap\Y)$, which is contained in $\supp\X\cap\supp\Y$.
We now conclude that the equality $\supp(\X\cap\Y)=\supp\X\cap\supp\Y$ holds.

(2) By (1) and Theorem \ref{23}  we have (a)$\Rightarrow$(b), the implication (b)$\Rightarrow$(c) is straightforward, and Proposition \ref{24} yields (c)$\Rightarrow$(a).
Hence the three conditions (a), (b) and (c) are equivalent.
\end{proof}

It is natural to consider the singularity category version of Question \ref{13}, i.e., the question asking when
\begin{equation}\label{26}
\ssupp(\X\cap\Y)=\ssupp\X\cap\ssupp\Y
\end{equation}
holds for thick subcategories $\X,\Y$ of $\ds(R)$.
The question admits a clean answer, when $R$ is locally dominant.

\begin{prop}\label{27}
Suppose that the ring $R$ is locally dominant (e.g., $R$ is locally a hypersurface).
Then the equality \eqref{26} holds true for all thick subcategories $\X,\Y$ of the singularity category $\ds(R)$.
\end{prop}

\begin{proof}
For all thick subcategories $\X,\Y$ of $\db(R)$ containing $R$, we shall prove the equality
$$
\ipd(\X\cap\Y)=\ipd(\X)\cap\ipd(\Y).
$$
Clearly, the right-hand side contains the left-hand side.
Pick any prime ideal $\p$ belonging to $\ipd(\X)\cap\ipd(\Y)$.
Note that $\p$ is in $\sing R$, and that $\ipd(\X),\ipd(\Y)$ are specialization-closed.
We see that $\ipd(R/\p)=\V(\p)$, and this is contained in both $\ipd(\X)$ and $\ipd(\Y)$.
Since the ring $R$ is locally dominant, $R/\p$ belongs to both $\X$ and $\Y$ by Theorem \ref{5}(1).
It follows that $\p\in\V(\p)=\ipd(R/\p)\subseteq\ipd(\X\cap\Y)$, and we are done.
\end{proof}

Using the above proposition, we get the corollary below, which says that if $R$ is a local complete intersection of codimension at least two, there exist thick subcategories $\X,\Y$ of $\ds(R)$ such that \eqref{26} does not hold.

\begin{cor}\label{28}
Suppose that $R$ is a local ring which is a complete intersection.
Then the equality \eqref{26} holds true for all thick subcategories $\X$ and $\Y$ of $\ds(R)$ if and only if the local ring $R$ is a hypersurface.
\end{cor}

\begin{proof}
By Proposition \ref{27}, the ``if'' part holds.
We now prove the ``only if'' part.
This is trivial in the case where $R$ is regular, so we may assume that $R$ is singular.
Take two nonzero thick subcategories $\X,\Y$ of $\ds(R)$.
Then the maximal ideal $\m$ of $R$ belongs to the singular supports $\ssupp\X$ and $\ssupp\Y$.
By assumption, the equality $\ssupp(\X\cap\Y)=\ssupp\X\cap\ssupp\Y$ holds.
Hence $\m$ belongs to $\ssupp(\X\cap\Y)$, and therefore, the intersection $\X\cap\Y$ is nonzero.
It is observed from \cite[Theorem 3.7]{core} that $R$ is a hypersurface.
\end{proof}

We close the section by posing the question below, which asks whether the converse of Proposition \ref{27} is true.
Note that Corollary \ref{28} supports this question in the affirmative, since a local complete intersection is dominant if and only if it is a hypersurface; see \cite[Theorem 1.1(4)]{dlr}.

\begin{ques}
Suppose \eqref{26} holds for all thick subcategories $\X,\Y$ of $\ds(R)$.
Is then $R$ locally dominant?
\end{ques}

\section{Characterizing thick subcategories satisfying the support equality}\label{8}

The purpose of this short section is to consider characterizing those thick subcategories $\X,\Y$ of $\db(R)$ which satisfy the equality \eqref{14}, and then restrict it to $\Y=\dpf(R)$.
We begin with stating a lemma.

\begin{lem}\label{29}
Let $R$ be local.
Let $X,Y\in\db(R)$ with $\rhom_R(X,Y)=0$.
Then one has $X=0$ or $Y=0$.
\end{lem}

\begin{proof}
It follows from \cite[Proposition 16.4.19]{CFH} that $\pd_R\rhom_R(X,Y)=\pd_RY+\sup X$, where $\sup X:=\sup\{i\in\Z\mid\h^i(X)\ne0\}$.
Since $\rhom_R(X,Y)=0$ and $\pd_R0=-\infty$ by the definition of projective dimension, either $\pd_RY$ or $\sup X$ is equal to $-\infty$.
We have $Y=0$ in the former case, and $X=0$ in the latter case.
\end{proof}

Now we are able to give a proof of the following theorem, which is the main result of this section.
Note that this theorem especially says that the equality \eqref{14} holds if $\Y$ is contained in $\dpf(R)$.

\begin{thm}\label{21}
Let $\X,\Y$ be thick subcategories of $\db(R),\dpf(R)$ respectively.
The following are equivalent.
\begin{enumerate}[\rm(1)]
\item
There is an equality $\supp(\X\cap\Y)=\supp\X\cap\supp\Y$.
\item
The extension $\X\ast\Y$ is a triangulated subcategory of $\db(R)$.
\item
Any morphism $f:X\to Y$ in $\db(R)$ with $X\in\X$ and $Y\in\Y$ factors through an object in $\X\cap\Y$.
\end{enumerate}
\end{thm}

\begin{proof}
The equivalence (2)$\Leftrightarrow$(3) is due to \cite[Lemma 1.1]{JK}.
The implication (1)$\Rightarrow$(3) is shown in a similar way as in the proof of \cite[Lemma 4.2(1)]{proxy}.
Let $f:X\to Y$ be a morphism in $\db(R)$ such that $X\in\X$ and $Y\in\Y$.
Take a set of generators $\xx=x_1,\dots,x_n$ of the annihilator of the $R$-module $\End_{\db(R)}(X)$.
Then there is an equality $\supp X=\V(\xx)$.
Let $K$ be the Koszul complex of $\xx$ over $R$, and put $Z=K\lten_RY$.
We have $Z\in\Y$ and $\supp K=\V(\xx)=\supp X$.
Therefore, it holds that
$$
\supp Z
=\supp(K\lten_RY)
=\supp K\cap\supp Y
=\supp X\cap\supp Y
\subseteq\supp\X\cap\supp\Y
=\supp(\X\cap\Y),
$$
where we use \cite[Corollary 15.1.17]{CFH} to get the second equality.
As $\Y$ is a subcategory of $\dpf(R)$, we have that $Z$ is in $\dpf(R)$ and $\X\cap \Y$ is a thick subcategory of $\dpf(R)$.
Theorem \ref{17} shows $Z\in\X\cap\Y$.
As $\xx$ annihilates $\End_{\db(R)}(X)$, there is an isomorphism $K\lten_RX\cong\bigoplus_{i=0}^n(X[i])^{\oplus\binom{n}{i}}$ in $\db(R)$.
We get a commutative diagram
$$
\xymatrix@R-1pc@C+5pc{
K\lten_RX\ar[d]^-{K\lten_Rf}\ar[r]^-\phi_-\cong&\bigoplus_{i=0}^n(X[i])^{\oplus\binom{n}{i}}\ar[r]^-\alpha& X[n]\ar[d]^-{f[n]}\\
K\lten_RY\ar[r]^-{\beta\lten_RY}& R[n]\lten_RY\ar[r]^-\psi_-\cong& Y[n]}
$$
in $\db(R)$, where $\alpha$ is the projection in degree $-n$, and $\beta:K\to R[n]$ is the morphism induced by truncation in degree $-n$.
Letting $\gamma:X[n]\to\bigoplus_{i=0}^n(X[i])^{\oplus\binom{n}{i}}$ be a morphism with $\alpha\gamma=1$, we obtain an equality
$$
f=(\psi\circ(\beta\lten_RY)\circ(K\lten_Rf)\circ\phi^{-1}\circ\gamma)[-n].
$$
Therefore, the morphism $f$ factors through the object $(K\lten_RY)[-n]=Z[-n]\in\X\cap\Y$.

Now, let us prove the remaining implication (3)$\Rightarrow$(1).
Assume $\supp\X\cap\supp\Y\ne\supp(\X\cap\Y)$.
Then $\supp\X\cap\supp\Y$ strictly contains $\supp(\X\cap\Y)$, and we find a prime ideal $\p$ of $R$ and objects $X,Y$ of $\db(R)$ such that $\p\in\supp X\cap\supp Y$ and $\p\notin\supp(\X\cap\Y)$.
Fix an integer $i$, and let $f:X\to Y[i]$ be a morphism in $\db(R)$.
By assumption, $f$ factors through an object $E\in\X\cap\Y$.
Localization at the prime ideal $\p$ shows that the morphism $f_\p:X_\p\to Y_\p[i]$ in $\db(R_\p)$ factors through the object $E_\p\in(\X\cap\Y)_\p$.
Since $\p$ is not in $\supp(\X\cap\Y)$, we get $E_\p=0$ and hence $f_\p=0$.
This shows $\Hom_{\db(R_\p)}(X_\p,Y_\p[i])=0$ for all $i\in\Z$.
Therefore, 
$$
\h^i(\rhom_{R_\p}(X_\p,Y_\p))=\Ext_{R_\p}^i(X_\p,Y_\p)=\Hom_{\db(R_\p)}(X_\p,Y_\p[i])=0
$$
for all integers $i$, which implies that $\rhom_{R_\p}(X_\p,Y_\p)=0$.
In view of Lemma \ref{29}, either $X_\p$ or $Y_\p$ is zero.
This contradicts the fact that $\p\in\supp X\cap\supp Y$.
We conclude that $\supp(\X\cap\Y)=\supp\X\cap\supp\Y$.
\end{proof}

Applying Theorem \ref{21} to $\Y=\dpf(R)$, we get a characterization of the proxy small subcategories.

\begin{cor}\label{2}
Let $\X$ be a thick subcategory of $\db(R)$.
The following statements are equivalent.
\begin{enumerate}[\rm(1)]
\item
The thick subcategory $\X$ of $\db(R)$ is proxy small.
\item
The extension $\X\ast\dpf(R)$ is a triangulated subcategory of $\db(R)$.
\item
Any morphism $f:X\to P$ in $\db(R)$ with $X\in\X$ and $P\in\dpf(R)$ factors through an object in $\X\cap\dpf(R)$.
\end{enumerate}
\end{cor}

\begin{rem}
The implication (1)$\Rightarrow$(3) in Corollary \ref{2} is none other than \cite[Lemma 4.2(1)]{proxy}.
\end{rem}

\section{Applications of Corollary \ref{2}}\label{19}

In this section, we shall provide applications of Corollary \ref{2}.
For this purpose, we introduce a notation.

\begin{dfn}\label{4}
Let $\E$ be a thick subcategory of $\dpf(R)$.
Let $\pi:\db(R)\to\db(R)/\E$ be the canonical functor.
We define the subcategory $\widetilde\E$ as follows; note that $\widetilde\E$ contains $\E$.
$$
\begin{array}{l}
\widetilde\E
=\{T\in\db(R)\mid\Hom_{\db(R)/\E}(\pi T,\pi \dpf(R))=0\}\\
\phantom{\widetilde\E}=\{T\in\db(R)\mid\Hom_{\db(R)/\E}(\pi T,\pi P)=0\text{ for all }P\in\dpf(R)\}.
\end{array}
$$
\end{dfn}

\begin{rem}\label{orth}
Let $\T$ be a triangulated category, and $\X$ a subcategory of $\T$.
The {\em left orthogonal} ${}^\perp\X=\{T\in\T\mid\Hom_\T(T,\X)=0\}$ of $\X$ in $\T$ is a subcategory of $\T$.
Note that ${}^{\perp}\X$ is thick when $\X$ is closed under shifts.
With the notation of Definition \ref{4}, one has the following equality, so that $\widetilde\E$ is a thick subcategory of $\db (R)$.
$$
\widetilde\E=\pi^{-1}({}^\perp(\pi\dpf(R))).
$$
\end{rem}

One can characterize the intermediate thick subcategories between a given thick subcategory $\E$ of $\dpf(R)$ and $\widetilde\E$.
It turns out that those thick subcategories are necessarily proxy small.

\begin{thm}\label{3}
Let $\E$ be a thick subcategory of $\dpf(R)$.
Then the following statements hold.
\begin{enumerate}[\rm(1)]
\item
$\widetilde\E$ is a proxy small subcategory of $\db (R)$ with $\widetilde\E\cap \dpf (R) =\E$ and $\supp\widetilde\E=\supp\E$. 
\item
The following are equivalent for each thick subcategory $\X$ of $\db(R)$.
\begin{enumerate}[\rm(a)]
\item
There are inclusions $\E\subseteq\X\subseteq\widetilde\E$.
\item
$\X$ contains $\E$ and satisfies the equality $\supp\X=\supp\E$.
\item
$\X$ is proxy small and satisfies the equality $\X\cap\dpf(R)=\E$.
\end{enumerate}
\end{enumerate}
\end{thm}

\begin{proof}
Let $\pi:\db(R)\to\db(R)/\E$ be the canonical functor.

(1) As we saw in Remark \ref{orth}, the subcategory $\widetilde\E$ of $\db(R)$ is thick and contains $\E$.
Pick any object $P\in\widetilde\E\cap\dpf(R)$.
It holds that $\Hom_{\db(R)/\E}(\pi P,\pi P)=0$, which implies that the object $\pi P$ of $\db(R)/\E$ is zero, so that $P$ belongs to $\E$.
It follows that $\widetilde\E\cap\dpf(R)=\E$.
Combining this with the equality $\Hom_{\db(R)/\E}(\pi\widetilde\E,\pi\dpf(R))=0$ given by the definition of $\widetilde\E$, we observe from \cite[Theorem A]{JK} and Corollary \ref{2} that the thick subcategory $\widetilde\E$ of $\db(R)$ is proxy small.
Hence we have the equalities $\supp\widetilde\E=\supp(\widetilde\E\cap\dpf(R))=\supp\E$.

(2)
(a)$\Rightarrow$(b):
If $\E \subseteq \X \subseteq \widetilde\E$, then $\supp\E\subseteq\supp\X\subseteq\supp\widetilde\E$, so that $\supp\X=\supp\widetilde\E=\supp\E$ by (1).

(b)$\Rightarrow$(c):
Suppose that $\X$ contains $\E$ and satisfies $\supp\X=\supp\E$.
The inclusions $\E\subseteq\X\cap\dpf(R)\subseteq\X$ give rise to inclusions $\supp\E\subseteq\supp(\X\cap\dpf(R))\subseteq\supp\X=\supp\E$.
It follows that $\supp(\X\cap\dpf(R))=\supp\X=\supp\E$.
Thus $\X$ is proxy small, and $\X\cap\dpf(R)=\E$ by Theorem \ref{17}.

(c)$\Rightarrow$(a):
Suppose that $\X$ is proxy small with $\X\cap\dpf(R)=\E$.
Then $\X$ contains $\E$.
Corollary \ref{2} and \cite[Theorem A]{JK} yield that $\Hom _{\db (R)/\E } (\pi\X, \pi\dpf (R))=0$, which says that $\X$ is contained in $\widetilde\E$.
\end{proof}

In the case of a locally dominant ring, the proxy small subcategories $\X$ of $\db(R)$ with $\X\cap\dpf(R)=\E$ bijectively correspond to the thick subcategories $\Y$ of $\db(R)$ with $R\in\Y$ and $\ipd(\Y)\subseteq\supp\E$.
Moreover, we can classify those proxy small subcategories.

\begin{thm}\label{6}
Let $\E$ be a thick subcategory of $\dpf(R)$.
Suppose that the ring $R$ is locally dominant.
Then there are mutually inverse bijections
\begin{equation}\label{30}
\xymatrix{
{\left\{\begin{matrix}
\text{specialization-closed}\\
\text{subsets $W$ of $\sing R$}\\
\text{with $W\subseteq\supp\E$}
\end{matrix}\right\}}
\ar@<.7mm>[r]^-t&
{\left\{\begin{matrix}
\text{proxy small subcategories}\\
\text{$\X$ of $\db(R)$}\\
\text{with $\X\cap\dpf(R)=\E$}
\end{matrix}\right\}}
\ar@<.7mm>[r]^-f\ar@<.7mm>[l]^-s&
{\left\{\begin{matrix}
\text{thick subcategories $\Y$}\\
\text{of $\db(R)$ with $R\in\Y$}\\
\text{and $\ipd(\Y)\subseteq\supp\E$}
\end{matrix}\right\}\!.}
\ar@<.7mm>[l]^-g}
\end{equation}
where the maps $s,t,f,g$ are given by
$$
\begin{array}{l}
s(\X)=\ipd(\X),\qquad
t(W)=\{X\in\db(R)\mid\supp X\subseteq\supp\E,\,\ipd(X)\subseteq W\},\\
f(\X)=\{Y\in\db(R)\mid\ipd(Y)\subseteq\ipd(\X)\},\qquad
g(\Y)=\{X\in\Y\mid\supp X\subseteq\supp\E\}.
\end{array}
$$
Moreover, one has $f(\X)=\thick_{\db(R)}(\X\cup\dpf(R))=\thick_{\db(R)}(\X\cup\{R\})$.
\end{thm}

\begin{proof}
Let $\pi:\db(R)\to\ds(R)$ stand for the canonical functor.
We divide the proof into three parts.

(1) Let us show that $s,t$ are mutually inverse bijections.
Fix a proxy small subcategory $\X$ of $\db(R)$ with $\X\cap\dpf(R)=\E$, and a specialization-closed subset $W$ of $\sing R$ contained in $\supp\E$.
Using Theorem \ref{3}(2), we see that $\ipd(\X)=\ssupp(\pi\X)\subseteq\supp\X=\supp\E$.
By the mutually inverse bijections $\sigma$ and $\tau$ given in Theorem \ref{5}(2), we have $\sigma(\X)=(\supp\E,s(\X))$ and $t(W)=\tau(\supp\E,W)$.
Also, the subcategory
$$
t(W)\cap\dpf(R)=\{X\in\dpf(R)\mid\supp X\subseteq\supp\E,\,\ipd(X)\subseteq W\}
$$
obviously contains $\E$, and is contained in $\E$ by Theorem \ref{17}.
We get $t(W)\cap\dpf(R)=\E$.
Thus, the maps $s,t$ are well-defined.
There are equalities
$$
\X=\tau\sigma(\X)=\tau(\supp\E,s(\X))=ts(\X),\quad
(\supp\E,st(W))=\sigma(t(W))=\sigma\tau(\supp\E,W)=(\supp\E,W).
$$
We obtain $ts=1$ and $st=1$.
Consequently, the maps $s$ and $t$ are mutually inverse bijections.

(2) Let us show that $f,g$ are mutually inverse bijections.
The mutually inverse bijections given in Theorem \ref{5}(1) induce a one-to-one correspondence between the thick subcategories $\Y$ of $\db(R)$ with $R\in\Y$ and $\ipd(\Y)\subseteq\supp\E$ and the specialization-closed subsets $W$ of $\sing R$ with $W\subseteq\supp\E$.
Splicing this one-to-one correspondence with the one given in (1), we get mutually inverse bijections
$$
\xymatrix{
{\left\{\begin{matrix}\text{proxy small subcategories $\X$}\\
\text{of $\db(R)$ with $\X\cap\dpf(R)=\E$}\end{matrix}\right\}}
\ar@<.7mm>[r]^-f& 
{\left\{\begin{matrix}
\text{thick subcategories $\Y$ of $\db(R)$}\\
\text{with $R\in\Y$ and $\ipd(\Y)\subseteq\supp\E$}
\end{matrix}\right\}\!,}
\ar@<.7mm>[l]^-h}
$$
where $h$ is the map defined by the equality $h(\Y)=\{X\in\db(R)\mid\supp X\subseteq\supp\E\text{ and }\ipd(X)\subseteq\ipd(\Y)\}$ for each thick subcategory $\Y$ of $\db(R)$ with $R\in\Y$ and $\ipd(\Y)\subseteq\supp\E$.
Applying Theorem \ref{5}(1) again, we easily observe that $h(\Y)=g(\Y)$.
Thus $h=g$.
We conclude that $f,g$ are mutually inverse bijections.

(3) Let us show the final assertion of the theorem.
Let $\X$ be a proxy small subcategory of $\db(R)$ with $\X\cap\dpf(R)=\E$.
Put $\ZZ=\thick_{\db(R)}(\X\cup\dpf(R))$.
Evidently, $\ZZ$ contains $R$.
It is easy to observe that
$$
\ipd(\ZZ)=\ipd(\X\cup\dpf(R))=\ipd(\X)\subseteq\supp\X=\supp\E.
$$
Thus $\ZZ$ belongs to the third set in the one-to-one correspondences \eqref{30}.
Since $f(\X)$ contains $\ZZ$, we have
$$
\X=gf(\X)\supseteq g(\ZZ)=\{C\in\ZZ\mid\supp C\subseteq\supp\E\}\supseteq\X.
$$
It follows that $gf(\X)=g(\ZZ)$.
By the injectivity of the map $g$, we obtain $f(\X)=\ZZ$.
\end{proof}

\begin{rem}
For $\E=\dpf(R)$, the second and third sets in the one-to-one correspondences \eqref{30} equal to the set of thick subcategories of $\db(R)$ containing $R$, and $f,g$ equal to the identity map by Theorem \ref{5}(1).
\end{rem}

Let $E$ be a subset of $\spec R$.
We denote by $\dbe(R)$ the subcategory of $\db(R)$ consisting of objects $X$ such that $\supp X$ is contained in $E$, and set $\dpfe(R)=\dpf(R)\cap\dbe(R)$.
We denote by $\dse(R)$ the subcategory of $\db(R)$ consisting of objects $X$ such that $\ssupp X$ is contained in $E$, where $\pi:\db(R)\to\ds(R)$ stands for the canonical functor.
It is possible to interpret the one-to-one correspondence in Theorem \ref{6}(2) as follows.

\begin{cor}
Let $R$ be a locally dominant ring.
Let $E$ be a specialization-closed subset of $\spec R$.
Then the assignments $\X\mapsto\thick_{\db(R)}(\X\cup\{R\})$ and $\Y\cap\dbe(R)\mapsfrom\Y$ give mutually inverse bijections
$$
{\left\{\begin{matrix}\text{thick subcategories $\X$ of $\db(R)$}\\
\text{with $\dpfe(R)\subseteq\X\subseteq\dbe(R)$}\end{matrix}\right\}}
\cong
{\left\{\begin{matrix}
\text{thick subcategories $\Y$ of $\db(R)$}\\
\text{with $\dpf(R)\subseteq\Y\subseteq\pi^{-1}\dse(R)$}
\end{matrix}\right\}\!.}
$$
\end{cor}

\begin{proof}
Fix thick subcategories $\X$ and $\Y$ of $\db(R)$.
Set $\E=\dpfe(R)$.
By the equivalence (b)$\Leftrightarrow$(c) in Theorem \ref{3}(2) and the one-to-one correspondence $(f,g)$ in Theorem \ref{6}, it suffices to show the following equivalences.
\begin{enumerate}
\item
$\E\subseteq\X\text{ and }\supp\X=\supp\E\iff\dpfe(R)\subseteq\X\subseteq\dbe(R)$.
\item
$R\in\Y\text{ and }\ipd(\Y)\subseteq\supp\E\iff\dpf(R)\subseteq\Y\subseteq\pi^{-1}\dse(R)$.
\end{enumerate}
Theorem \ref{17} implies that there is an equality $\supp\E=E$.
When $\X$ contains $\E$, the equivalences below hold.
$$
\supp\E=\supp\X
\iff\supp\E\supseteq\supp\X
\iff E\supseteq\supp\X
\iff\X\subseteq\dbe(R).
$$
Thus (1) follows.
Since $\Y$ is thick, it contains $R$ if and only if it contains $\dpf(R)$.
There are equivalences
$$
\Y\subseteq\pi^{-1}\dse(R)
\iff\ssupp(\pi\Y)\subseteq E
\iff\ipd(\Y)\subseteq E
\iff\ipd(\Y)\subseteq\supp\E,
$$
where $\pi:\db(R)\to\ds(R)$ stands for the canonical functor.
Now we observe that (2) holds.
\end{proof}

\begin{ac}
The authors thank Hiroki Matsui for asking the second author whether the condition \cite[Lemma 4.2(1)]{proxy} characterizes proxy smallness; this is answered in the affirmative in Corollary \ref{2}.
\end{ac}


\end{document}